\documentclass[12pt]{amsart}

\usepackage{graphicx}
\usepackage{amssymb}
\usepackage{float}
\usepackage{fullpage}

\usepackage{comment}
\usepackage{physics}

\newtheorem{theorem}{Theorem}[section]
\newtheorem{lemma}[theorem]{Lemma}
\newtheorem{corollary}[theorem]{Corollary}
\newtheorem{proposition}[theorem]{Proposition}
\newtheorem{conjecture}[theorem]{Conjecture}

\theoremstyle{definition}
\newtheorem{definition}[theorem]{Definition}

\theoremstyle{remark}
\newtheorem{remark}[theorem]{Remark}

\numberwithin{equation}{section}

\newcommand{\R}{{\mathbb R}}
\newcommand{\Q}{{\mathbb Q}}
\newcommand{\Z}{{\mathbb Z}}
\newcommand{\N}{{\mathbb N}}
\newcommand{\li}{\textnormal{li}}

\newcommand{\eps}{{\epsilon}}

\begin{document}

\title[The Erd\H os primitive set conjecture]{A proof of the Erd\H os primitive set conjecture}

\author{Jared Duker Lichtman}
\address{Mathematical Institute, University of Oxford, Oxford, OX2 6GG, UK}

\email{jared.d.lichtman@gmail.com}


\dedicatory{Dedicated to Carl Pomerance}

\begin{abstract}
A set of integers greater than 1 is primitive if no member in the set divides another. Erd\H{o}s proved in 1935 that the series $f(A) = \sum_{a\in A}1/(a \log a)$ is uniformly bounded over all choices of primitive sets $A$. Later he asked if this bound is attained for the set of prime numbers. In this article we answer in the affirmative.

As further applications of the method, we make progress towards a question of Erd\H{o}s, S\'ark\"ozy, and Szemer\'edi from 1968. We also refine the classical Davenport--Erd\H{o}s theorem on infinite divisibility chains, and extend a result of Erd\H{o}s, S\'ark\"ozy, and Szemer\'edi from 1966.

\end{abstract}



\maketitle

\section{Introduction}

A set of integers $A\subset \mathbb{Z}_{>1}$ is {\it primitive} if no member in $A$ divides another. For example, the integers in a dyadic interval $(x,2x]$ form a primitive set. Similarly the set of primes is primitive, along with the set $\mathbb{N}_k$ of numbers with exactly $k$ prime factors (with multiplicity), for each $k\ge1$. Another well-known example is the set of perfect numbers.\footnote{
Since Ancient Greece, a number $n$ is classified as `perfect,' `abundant,' or `deficient,' depending on whether the sum of its proper divisors equals $n$, is greater than $n$, or is less than $n$, respectively.}

The study of primitive sets emerged in the 1930s as a generalization of one special problem. A classical theorem of Davenport asserts that the set of abundant numbers has a positive asymptotic density. This was originally proved by sophisticated analytic methods, but Erd\H{o}s soon found an elementary proof by using primitive abundant numbers.\footnote{More precisely, `primitive non-deficient numbers'} The proof ideas led people to introduce the abstract definition of primitive sets and study them for their own sake. See Hall \cite{Hsetmult} or Halberstam--Roth \cite[\S 5]{HalbRoth} for detailed introductions to the subject.

There are a number of interesting and sometimes unexpected theorems about primitive sets. For instance, in 1934 Besicovitch \cite{Besicovitch} showed that the upper asymptotic density of a primitive set can be arbitrarily close to $1/2$, whereas in 1935 Behrend \cite{Behrend} and Erd\H{o}s \cite{Erdos35} proved the lower asymptotic density is always $0$. In fact, Erd\H{o}s proved the stronger result that
\begin{align*}
f(A) := \sum_{a\in A}\frac{1}{a\log a} \ < \ \infty,
\end{align*}
uniformly over all primitive sets $A$. Later Erd\H{o}s famously asked if the maximum is attained by the primes $\mathcal{P}$. This appears at least since 1974 \cite[p.11]{ErdosIndian}.

\begin{conjecture}[Erd\H{o}s primitive set conjecture] \label{conj:EPS}
For any primitive set $A$, \
$f(A) \le f(\mathcal P)$.
\end{conjecture}


The prime sum is $f(\mathcal{P}) = \sum_p 1/(p\log p)=1.6366\cdots$ after computations of Cohen \cite{Cohen}. In 1993, Erd\H{o}s and Zhang \cite{EZ} proved the bound $f(A) < 1.84$ for all primitive $A$. Recently in 2019, Lichtman and Pomerance \cite{LPprim} improved the bound to $f(A) < e^\gamma=1.781\cdots$, where $\gamma$ is the Euler-Mascheroni constant. Note the tail of the series for $f(\mathcal{P})$ converges quite slowly $O(1/\log x)$, and moreover there are sets $A\subset[x,\infty)$ for which $f(A)\sim 1$ as $x\to\infty$ (in this connection see Conjecture \ref{conj:ESS} below). As such, Conjecture \ref{conj:EPS} is not susceptible to direct attack by computing partial sums up to $x$.

One potential strategy to approach Conjecture \ref{conj:EPS} is via integration. Namely,
\begin{align*}
f(A) = \sum_{a\in A}\frac{1}{a\log a} = \sum_{a\in A}\int_1^\infty a^{-t}\dd{t} = \int_1^\infty f_t(A)\dd{t},
\end{align*}
letting $f_t(A) = \sum_{a\in A} a^{-t}$. So Conjecture 1 would follow if for any $t>1$, primitive set $A$,
\begin{align}\label{eq:ftA}
f_t(A) \ \le \ f_t(\mathcal P).
\end{align}
However, it was shown in \cite{BM} that \eqref{eq:ftA} holds if and only if
\begin{align*}
t \ge \tau:=1.1403\cdots,
\end{align*}
where $t=\tau$ is the unique real solution to the equation
\begin{align*}
\sum_p p^{-t} = 1+\Big(1-\sum_p p^{-2t}\Big)^{1/2}.
\end{align*}
The fact that $\tau$ is markedly larger than 1 gives some indication as to why the Erd\H{o}s primitive set conjecture has remained open. 

Similar analysis actually enables a disproof of a natural analogue of Conjecture \ref{conj:EPS} for the translated sum $f(A,h)=\sum_{a\in A}1/a(\log a+h)$, in that there are primitive $A$ for which $f(A,h)>f(\mathcal P,h)$ once $h\ge81$ \cite{LDM}, \cite{Laib}. This was refined down to just $h\ge1.04$ in \cite{Ltrans}, and suggests that the original conjecture (when $h=0$), if true, is only `barely' so. 

Concerning \eqref{eq:ftA}, we also note Chan et al. \cite{CLP2prim} proved $f_t(A) \ \le \ f_t(\mathcal P)$ for all $t\ge .7983$ for all 2-primitive sets $A$, thereby resolving Conjecture 1 in this special case (also see \cite{CLPkprim}). Here a set $A$ is 2-primitive if no member of $A$ divides the product of 2 others.

A separate strategy for the problem is to split up $A$ according to the smallest prime factor. That is, for each prime $p$ let
$$A_p = \{ n\in A : n \text{ has least prime factor }p\}.$$
As in \cite{LPprim}, we say $p$ is {\it Erd\H{o}s strong} if the singleton set $\{p\}$ maximizes $f(A)$ among all primitive sets $A$ all of whose elements have least prime factor $p$. That is, $f(A_p)\le f(\{p\})=:f(p)$ for all primitive $A$. Conjecture \ref{conj:EPS} would follow if every prime is Erd\H{o}s strong, since then $f(A) = \sum_p f(A_p) \le f(\mathcal{P})$.

By a short argument in \cite{LPprim} (also see Lemma \ref{lem:LMertfA}), a sufficient condition for a prime $p$ to be Erd\H{o}s strong is that
\begin{align}\label{eq:Mertprim}
e^{\gamma}\prod_{q<p}\Big(1-\frac{1}{q}\Big) \ \le \ \frac{1}{\log p}.
\end{align}
Here $q$ runs over primes. Note the two sides of this inequality are asymptotically equal by Mertens' prime product theorem. By direct computation, \eqref{eq:Mertprim} is satisfied by the first $10^8$ odd primes, but fails for $p=2$ since $\log 2> e^{-\gamma}$.

Moreover $99.999973\%$ of primes\footnote{More precisely, the set of such primes has discrete, logarithmic density equal to $0.99999973\cdots $ within $\mathcal P$.} satisfy \eqref{eq:Mertprim}, assuming the Riemann Hypothesis and the Linear Independence hypothesis\footnote{Namely, the sequence of numbers $\gamma_n > 0$ such that $\zeta(\frac{1}{2} + i\gamma_n) = 0$ is linearly independent over $\Q$.} \cite{LPrace}. This result is intimately related to the celebrated work of Rubinstein and Sarnak \cite{RubSarn} on the prime number race between $\pi(x)$ and $\li(x)$. On the Riemann Hypothesis alone \eqref{eq:Mertprim} fails for a positive proportion of primes $p$ (in log density), and even unconditionally \eqref{eq:Mertprim} is known to fail for infinitely many primes $p$. This perhaps suggests Conjecture \ref{conj:EPS} might be false, or at least beyond the reach of unconditional tools.


In this article we establish Conjecture \ref{conj:EPS}.
\begin{theorem}\label{thm:EPS}
For any primitive set $A$, we have
$f(A) \le f(\mathcal P)$.
\end{theorem}

Moreover, we show that every odd prime is Erd\H{o}s strong.
\begin{theorem}\label{thm:Estrongodd}
For any primitive set $A$ and any prime $p>2$, we have $f(A_p) \le f(p)$.
\end{theorem}

It remains an open question whether $p=2$ is Erd\H{o}s strong.

Another question related to Conjecture \ref{conj:EPS}, in 1968 Erd\H{o}s, S\'ark\"ozy, and Szemer\'edi posed the following \cite[eq. (11)]{ESS68}.

\begin{conjecture}[Erd\H{o}s--S\'ark\"ozy--Szemer\'edi] \label{conj:ESS}
We have
\begin{align*}
\lim_{x\to\infty}\sup_{\substack{A\subset[x,\infty)\\A\textnormal{ primitive}}}f(A) \ \le \ 1.
\end{align*}
\end{conjecture}

This also appears in \cite[p.\,244]{MathErdos} as Problem 2.2, and in \cite[p.\,224]{MathErdosII} as Problem 2.

Not much has been proven in this direction until very recently. Recall the set $\mathbb{N}_k$ of numbers with exactly $k$ prime factors (with multiplicity) lies in $[2^k,\infty)$. Lichtman and Pomerance \cite{LPprim} proved $f(\N_k)\gg 1$, and in \cite{Lalmost} it was shown $f(\N_k)\sim 1$ as $k\to\infty$. This means that if Conjecture \ref{conj:ESS} holds, then the limit must attain an {\it equality} of 1. We note \cite[Theorem 4.1]{Lalmost} gives for all $\eps>0$,
\begin{align}\label{eq:fNksim1}
f(\N_k) = 1 + O_\eps(k^{\eps-1/2}).
\end{align}
Moreover, computations up to $k=20$ suggest the true rate of decay may be exponential $O(2^{-k})$, see \cite{Lalmost}.

The methods in this paper enable the following progress towards Conjecture \ref{conj:ESS}.
\begin{theorem}\label{thm:ESS}
We have
\begin{align*}
\lim_{x\to\infty}\sup_{\substack{A\subset[x,\infty)\\A\textnormal{ primitive}}}f(A) \ \le \ e^\gamma\frac{\pi}{4} \approx 1.399.
\end{align*}
\end{theorem}

\subsection*{Notation}
Let $p(a),P(a)$ denote the smallest and largest prime factors of $a\in\Z_{>1}$, respectively, and denote $a^*=a/P(a)$. Let $\Omega(n)$ denote the number of prime factors of $n$ (with multiplicity) and let $\N_k = \{n: \Omega(n)=k\}$. Define $f(a) = 1/(a\log a)$ and $f(A) = \sum_{a\in A}f(a)$ for $A\subset \Z_{>1}$. Let $\mathcal P$ be the set of prime numbers, whose elements we denote by $p$ and $q$, unless otherwise stated. Also $p^k\| n$ means $p^k\mid n$ and $p^{k+1}\nmid n$.

\subsection{Proof outline of Theorem \ref{thm:EPS}}

The proof is a refinement of the argument of \cite{LPprim}. The key new idea is to exploit the fact that $A$ cannot contain too many elements $a$ with $P(a)$ just slightly less than $a$. This improves the critical case in the argument of \cite{LPprim}, and ultimately leads to an improvement by a factor of $\pi/4$ from a contribution from each $a\in A$ which is not prime. Since $e^\gamma\pi/4<f(\mathcal P)$, this ultimately means that $f(A)$ is maximized when all elements are prime. (Additional care is needed for small numbers, using explicit bounds.)

Let us recall the rough argument of \cite{LPprim} (suppressing details for primes and small numbers). By Mertens' product theorem,
\begin{align}\label{eq:sumofden}
f(A)=\sum_{a\in A}\frac{1}{a\log{a}} < \sum_{a\in A}\frac{1}{a\log P(a)}\approx  e^\gamma\sum_{a\in A}\frac{1}{a}\prod_{p<P(a)}\Big(1-\frac{1}{p}\Big).
\end{align}
But $a^{-1}\prod_{p<P(a)}(1-p^{-1})$ is the natural density of ${\rm L}_a=\{ba\;:\;p\mid b \Rightarrow p\ge P(a)\}$, and these sets turn out to be disjoint by primitivity of $A$ (Lemma \ref{lem:trichot}). So the sum of densities in \eqref{eq:sumofden} is trivially at most 1, leading to the bound $f(A)<e^\gamma$ for primitive $A$. This is inspired by the original 1935 argument of Erd\H{o}s \cite{Erdos35}.

There is a loss in the above argument when bounding $a$ by $P(a)$, and this loss is largest when $a$ is far from prime. We can save an additional factor of $\log{P(a)}/\log a$ for any individual $a\in A$, and this would be a significant improvement in the case $P(a)^2<a$, say. Therefore the critical case to handle is when $a\in A$ is composite with $P(a)$ close to $a$ in size. The key new ingredient (Proposition \ref{prop:maindensity}) shows that if $P(a)^{1+v}>a$ uniformly for all $a\in A$ (so the savings factor is $\log P(a)/\log a > 1/(1+v)$), then we can bound the sum of densities in \eqref{eq:sumofden} by $\sqrt{v}$. This refines the trivial bound of 1 in the range $0<v<1$, and quantifies the earlier statement that $A$ contains few elements $a$ with $P(a)$ slightly less than $a$. As the savings $1/(1+v)$ improves with $v$, the worst-case scenario is when the subset of $a\in A$ with $P(a)^{1+v}\approx a$ contributes about $\frac{\dd}{\dd v}[\sqrt{v}]=1/2\sqrt{v}$ to the sum of densities in \eqref{eq:sumofden}. Combining these ingredients ultimately leads to a savings of $\int_0^1\dd v/2\sqrt{v}(1+v)=\pi/4$, as desired.

Lastly, the key Proposition \ref{prop:maindensity} relies on the following observation (Lemma \ref{lem:disjointLac}): not only are the sets ${\rm L}_a$ disjoint, but so too are ${\rm L}_{ac}$ for many choices of integers $c$ (in fact all choices of $c$ with prime factors between $P_2(a)$ and $P_2(a)^{1/\sqrt{v}})$. Thus the sum of densities of these ${\rm L}_{ac}$ must be at most 1. But these sets ${\rm L}_{ac}$ are self-similar to the ${\rm L}_a$, and so the sum of their densities is roughly $1/\sqrt{v}$ times that of the ${\rm L}_a$, giving the desired bound $\sqrt{v}$.

\subsection{{\rm L}-primitive sets}

As outlined above, the subset of multiples of each $a\in A$,
\begin{align}\label{eq:defLa}
{\rm L}_a & := \big\{ba\in \N \;:\; p\mid b \implies p\ge P(a) \big\},
\end{align}
arises naturally in our proof. As such we shall introduce `{\rm L}' refinements of our common notions (here {\rm L} alludes to `lexicographic'). Specifically, if $n\in {\rm L}_a$ we say $n$ is an {\it {\rm L}-multiple} of $a$, and $a$ is an {\it{\rm L}-divisor} of $n$. Most importantly, we introduce the following key definition.
\begin{definition}
A set $A\subset \Z_{>1}$ is {\it {\rm L}-primitive} if $a'\notin {\rm L}_a$ for all distinct $a,a'\in A$.
\end{definition}
That is, $A$ is {\rm L}-primitive if no member of $A$ is an {\rm L}-multiple of another. In particular this definition is weaker than primitive. 

One may apply the basic argument as in \eqref{eq:sumofden} more generally for {\rm L}-primitive sets $A$, leading to the same bound $f(A) < e^\gamma$ (again ignoring small numbers). Moreover, {\rm L}-primitive sets play a central role in the proof of Theorem \ref{thm:EPS}. However, it turns out the bound $e^\gamma$ is essentially best possible for {\rm L}-primitive sets (see Proposition \ref{prop:Lprimsup}), which is markedly larger than $f(\mathcal P)$. This further highlights the subtlety of Conjecture \ref{conj:EPS}.

\subsection{Density and divisibility chains}
Recall the natural (asymptotic) density ${\rm d}(S)=\lim_{x\to\infty} |S\cap [1,x]|/x$ of a set $S\subset \N$. We also consider $\log$ density $\delta(S)$ and $\log\log$ density $\Delta(S)$, given by
\begin{align}\label{eq:densitydefs}
\delta(S) & = \lim_{x\to\infty} \frac{1}{\log x}\sum_{n\in S, n\le x}\frac{1}{n},\qquad\text{and}\quad
&\Delta(S)  = \lim_{x\to\infty} \frac{1}{\log\log x}\sum_{n\in S, 1<n\le x}\frac{1}{n\log n},
\end{align}
provided these limits exist. Recall the corresponding upper densities $\overline{\rm d}(S)$, $\overline{\delta}(S)$, $\overline{\Delta}(S)$ always exist, by replacing $\lim_{x\to\infty}$ with $\limsup_{x\to\infty}$ (and similarly $\liminf_{x\to\infty}$ for lower densities).

Taking an abstract view, a primitive set is an antichain for the partial ordering of integers by divisibility. As such this naturally leads to the dual notion of a chain in this context. Namely, an infinite sequence of integers $1<d_1<d_2<\cdots$ is a {\it divisibility chain} if $d_j\mid d_{j+1}$ for all $j\ge1$. A classical 1937 theorem of Davenport and Erd\H{os} \cite{DE37} asserts that if set $A\subset \N$ has upper $\log$ density $\overline{\delta}(A)>0$, then it contains an infinite divisibility chain $D\subset A$.

Analogously, we introduce the following refinement.

\begin{definition}
An infinite sequence of integers $1<d_1<d_2<\cdots$ is an {\it{\rm L}-divisibility chain} if $d_{j+1}\in {\rm L}_{d_j}$ for all $j\ge1$.
\end{definition}

That is, $d_{j+1}$ is an {\rm L}-multiple of $d_j$ for all $j\ge1$. In particular this definition is stronger than a (mere) divisibility chain. 

We refine the Davenport--Erd\H{os} theorem to {\rm L}-divisibility chains.
\begin{theorem}\label{thm:LDavenErdos}
If a set $A\subset \N$ has upper $\log$ density $\overline{\delta}(A)>0$, then $A$ contains an infinite {\rm L}-divisibility chain.
\end{theorem}

In 1966 Erd\H{o}s, S\'ark\"ozy, and Szemer\'edi \cite[Theorem 1]{ESS66} quantified the Davenport--Erd\H{os} theorem by showing such a divisibility chain $D$ satisfies $\limsup_{y\to\infty}\sum_{d\in D,d\le y}1/\sqrt{\log\log y} > 0$, and proved such growth rate is best possible.

They also studied the analogous question for upper $\log\log$ density, which they write ``seems more interesting to us.'' Namely, in \cite[Theorem 2]{ESS66} they established the following quantitative result.

\begin{theorem}[Erd\H{o}s--S\'ark\"ozy--Szemer\'edi] \label{thm:ESSchain}
If $A\subset \N$ has upper $\log\log$ density $\overline{\Delta}(A)>0$, then there is an infinite divisibility chain $D\subset A$ of growth
\begin{align}\label{eq:ESSloglog}
\limsup_{y\to\infty}\sum_{\substack{d\in D\\d\le y}}\frac{1}{\log\log y} \ \ge \ \frac{\overline{\Delta}(A)}{e^\gamma}.
\end{align}
\end{theorem}

Analogously, we quantify Theorem \ref{thm:LDavenErdos} in the case of $\log\log$ density, thereby refining Theorem \ref{thm:ESSchain} of Erd\H{o}s--S\'ark\"ozy--Szemer\'edi to {\rm L}-divisibility chains.

\begin{theorem}\label{thm:ESSLchain}
If $A\subset \N$ has upper $\log\log$ density $\overline{\Delta}(A)>0$, then there is an infinite {\rm L}-divisibility chain $D\subset A$ of growth
\begin{align*}
\limsup_{y\to\infty}\sum_{\substack{d\in D\\d\le y}}\frac{1}{\log\log y} \ \ge \ \frac{\overline{\Delta}(A)}{e^\gamma}.
\end{align*}
\end{theorem}

In view of Proposition \ref{prop:Lprimsup} we believe that the lower bound $\overline{\Delta}(A)/e^\gamma$ above is best possible for {\rm L}-divisibility chains, though we are unable to settle this. Notably, this contrasts the situation in Theorem \ref{thm:ESSchain}, as Erd\H{o}s--S\'ark\"ozy--Szemer\'edi conjectured $\overline{\Delta}(A)/e^\gamma$ in \eqref{eq:ESSloglog} might be improved to $\overline{\Delta}(A)$, which would be best possible for divisibility chains, if true \cite[eq. (5)]{ESS66}.

\section{Preliminaries on {\rm L}-primitive sets}

Recall the set of {\rm L}-multiples ${\rm L}_a:=\{ba\in\N \,:\,p\mid b \Rightarrow p\ge P(a)\}$ from \eqref{eq:defLa}. In particular $a\in {\rm L}_a$ for $b=1$, and $p(b) \ge P(a)$ for $b>1$. For $A\subset \N$ define ${\rm L}_A :=\bigcup_{a\in A}{\rm L}_a$. Also let $A_a =A\cap {\rm L}_a$ so that $\N_a={\rm L}_a$ and $A_q = \{a\in A : p(a)=q\}$ for prime $q$.\footnote{Note the notation for $A_q$ differs slightly from what is used in \cite{EZ}, \cite{LPprim}}

Observe that $a\in {\rm L}_{a'}$ if and only if ${\rm L}_a\subset {\rm L}_{a'}$, as well as the following trichotomy.

\begin{lemma}\label{lem:trichot}
For any integers $a,a'>1$, if ${\rm L}_a\cap {\rm L}_{a'}\neq\emptyset$ then $a\in {\rm L}_{a'}$ or $a'\in {\rm L}_a$. Thus ${\rm L}_a\cap {\rm L}_{a'}=\emptyset$ or ${\rm L}_a\subset {\rm L}_{a'}$ or ${\rm L}_a\supset {\rm L}_{a'}$.
\end{lemma}
\begin{proof}
Suppose $ba=b'a' \in {\rm L}_a\cap {\rm L}_{a'}$. If $b=1$ or $b'=1$ then $a\in {\rm L}_{a'}$ or $a'\in {\rm L}_a$. Otherwise $b,b'>1$, so $P(a) \le p(b)$ and $P(a')\le p(b')$ imply $b\mid b'$ or $b'\mid b$. Thus $a' = a(b/b')\in {\rm L}_a$ or $a = a'(b'/b)\in {\rm L}_{a'}$ as well.
\end{proof}

As such we see $A$ is {\rm L}-primitive if and only if the sets $\{{\rm L}_a\}_{a\in A}$ are pairwise disjoint.

\begin{corollary}\label{cor:Lsplit}
If $A$ is an {\rm L}-primitive set, then ${\rm L}_a$ and ${\rm L}_{a'}$ are disjoint for distinct $a,a'\in A$.
\end{corollary}

Recall ${\rm L}_a$ has natural density ${\rm d}({\rm L}_a) = \frac{1}{a}\prod_{p<P(a)}(1-\frac{1}{p})$. And by Mertens' product theorem $\prod_{p<x}(1-\frac{1}{p})\sim 1/e^\gamma \log x$, where $\gamma=.57721\cdots$ is the Euler-Mascheroni constant. By a show argument below, we relate $f(A)$ to density of {\rm L}-multiples. This is essentially based on Erd\H{o}s \cite{Erdos35} (also see \cite[Lemma 1]{ESS67}, \cite[Proposition 2.1]{LPprim}).

\begin{lemma}\label{lem:LMertfA}
For an {\rm L}-primitive set $A$ and an integer $1<n\notin A$, we have $f(A_n) < e^\gamma\,{\rm d}({\rm L}_n)$.
\end{lemma}
\begin{proof}
We may assume $A=A_n$ is finite, since $f(A) = \lim_{x\to\infty} f(A\cap[1,x])$. As $n\notin A$ all elements of $A$ are composite. Also $A$ is {\rm L}-primitive so ${\rm d}({\rm L}_A)=\sum_{a\in A}{\rm d}({\rm L}_a)$ by Corollary \ref{cor:Lsplit}. 
Next, Theorem 7 in \cite{RS1} implies $\prod_{p<x}\frac{p}{p-1} < e^\gamma\log(2x)$ for all $x>1$. Thus for any composite integer $a>1$, we have $a>2P(a)$ so that
\begin{align*}
f(a) = \frac{1}{a\log a} \ \le \ \frac{1}{a\log 2P(a)} \ < \ \frac{e^\gamma}{a}\prod_{p<P(a)}\Big(1-\frac{1}{p}\Big) \ = \ e^\gamma\,{\rm d}({\rm L}_a).
\end{align*}
Hence $f(A) = \sum_{a\in A}f(a) < e^\gamma\,{\rm d}({\rm L}_A) \le e^\gamma\,{\rm d}({\rm L}_n)$ since $A\subset {\rm L}_n$.
\end{proof}

We shall also need a technical refinement of Lemma \ref{lem:LMertfA}. For this, we rewrite Mertens' product theorem as $\mu_x\sim 1$, where we denote
\begin{align}\label{eq:muxdef}
\mu_x := e^\gamma \log x\prod_{p<x}\Big(1-\frac{1}{p}\Big).
\end{align}
In particular, for a prime $q$ we have
\begin{align}\label{eq:fqdLq}
f(q) = \frac{1}{q\log q} = \frac{1}{q}\frac{e^\gamma}{\mu_q}\prod_{p<q}\Big(1-\frac{1}{p}\Big) =\frac{e^\gamma}{\mu_q}{\rm d}({\rm L}_q).
\end{align}
We have the following explicit bounds for $\mu_x$, which critically are monotonic. We give upper bounds which hold on real $x\in \R$, but for lower bounds it turns out it suffices to restrict to the subsequence of primes $q\in\mathcal P$.
\begin{lemma}[Monotonic bounds] \label{lem:monotonic}
For $q\in\mathcal P$ and $x\in \R$, define
\begin{align*}
m_q := \inf_{\substack{p\ge q\\ p\in \mathcal P}}\mu_p, \quad\text{and}\quad
M_x := \sup_{\substack{y\ge x\\ y\in \R}}\mu_y.
\end{align*}
Then we have
\[m_q \ge \begin{cases}
\mu_7 = 0.9242\cdots & q \le 7\\
\mu_{19} = 0.9467\cdots & 7< q \le 300\\
1- \frac{1}{2(\log q)^2} & q> 300.
\end{cases}
\quad\text{and}\quad
M_x \le \begin{cases}
\mu_2 = 1.235\cdots & x \le 2\\
1 + \frac{1}{2\log(2\cdot10^9)^2} & 2< x \le 2\cdot10^9\\
1 + \frac{1}{2(\log x)^2} & x> 2\cdot10^9.
\end{cases}
\]
\end{lemma}
\begin{proof}
First, Rosser-Schoenfeld \cite[Theorem 7]{RS1} implies the product over primes $p< x$ is bounded in between
\begin{align*}
1-\frac{1}{2(\log x)^2} \ \overset{(x>285)}{\le} \ e^\gamma \log x\prod_{p< x}\Big(1-\frac{1}{p}\Big)
\ \overset{(x>1)}{\le} \ 1+\frac{1}{2(\log x)^2}.
\end{align*}

Note $\mu_x$ is increasing on $x\in(p,p']$ for consecutive primes $p,p'$. So the upper bound follows by computing $\mu_p<1$ for the first $10^8$ odd primes $p$ (note $p_{10^8} \ge 2\cdot10^9$). Hence $\mu_x<1$ for real $2< x\le 2\cdot10^9$. Below we display $\mu_q$ for the first few primes $q$, rounded to 4 significant digits.
\[\begin{array}{cc|cc|cc|cc|cc}
q & \mu_q & q & \mu_q & q & \mu_q & q & \mu_q & q & \mu_q\\
\hline
2  &  1.235  & 31  &{\bf0.9660}& 73  &  0.9766  & 127  &  0.9902  & 179  &  0.9909  \\
3  &  0.9784  & 37  &  0.9831  & 79  &  0.9809  & 131  &  0.9887  & 181  &  0.9874  \\ 
5  &  0.9555  & 41  &  0.9836  & 83  &  0.9795  & 137  &  0.9902  & 191  &  0.9921  \\
7 &{\bf0.9242}& 43  &  0.9720  & 89  &  0.9829  & 139  &  0.9858  & 193  &  0.9889  \\ 
11  &  0.9762 & 47  &{\bf0.9718}  & 97  &  0.9906  & 149  &  0.9925  & 197  &  0.9876  \\ 
13  &  0.9492 & 53  &  0.9808  & 101  &  0.9890  & 151  &  0.9885  & 199  &{\bf 0.9844}  \\ 
17  &  0.9679 & 59  &  0.9883  & 103  &  0.9834  & 157  &  0.9896  & &  \\
19  &{\bf0.9467} & 61  &  0.9795  & 107  &  0.9818  & 163  &  0.9906  & & \\
23 &{\bf0.9551}& 67 &  0.9854  & 109  &  0.9765  & 167  &  0.9892  &  & \\
29  &  0.9811 & 71  &  0.9841  & 113 &{\bf0.9749}& 173  &  0.9900  &  & 
\end{array}\]
The lower bound follows by identifying the primes $q$ for which $\mu_q = \inf_{p\ge q} \mu_p$ (in bold above), and then computing $\mu_{199} < \mu_p$ for $199<p\le 300$, as well as checking $\mu_{199} < 0.9846 < 1-\frac{1}{2(\log x)^2}$ for $x>300$. (In practice we shall only need $\mu_q$ for $q=7,19$.)
\end{proof}

We may now prove a technical refinement of Lemma \ref{lem:LMertfA} using $\mu_q$.

\begin{lemma}\label{lem:fMertPnu}
Let $A$ be an {\rm L}-primitive set. Take $v\ge0$, an integer $n\notin A$, and denote $q=P(n)$. If $P(a)^{1+v}\le a$ for all $a\in A_n$, then
\begin{align*}
f(A_n) = \sum_{a\in A_n}\frac{1}{a\log a} \ &\le \ \frac{e^{\gamma}}{m_q}\,\frac{{\rm d}({\rm L}_{A_n})}{1+v}.
\end{align*}
\end{lemma}
\begin{proof}
We may assume $A = A_n$ is finite, since $f(A) = \lim_{x\to\infty} f(A\cap[1,x])$. As $n\notin A$ all elements of $A$ are composite. Also $A$ is {\rm L}-primitive so ${\rm d}({\rm L}_A)=\sum_{a\in A}{\rm d}({\rm L}_a)$ by Corollary \ref{cor:Lsplit}. Moreover $(1+v)\log P(a) \le \log a$ for all $a\in A$. Thus by definition of $\mu_{P(a)}$ in \eqref{eq:muxdef},
\begin{align*}
\frac{1}{a\log a} \ \le \ \frac{1}{1+v}\frac{1}{a\log P(a)} = \frac{e^\gamma}{\mu_{P(a)}}\frac{1}{(1+v)a}\prod_{p<P(a)}\Big(1-\frac{1}{p}\Big) = \frac{e^\gamma}{\mu_{P(a)}}\frac{{\rm d}({\rm L}_a)}{1+v}.
\end{align*}

By monotonicity $\mu_{P(a)}\ge m_{P(a)}\ge m_q$ for $a\in A\subset {\rm L}_n$. Hence we conclude
\begin{align*}
f(A) = \sum_{a\in A}\frac{1}{a\log a} \ \le \ \frac{e^\gamma}{m_q}\frac{1}{1+v}\sum_{a\in A}{\rm d}({\rm L}_a) = \frac{e^\gamma}{m_q}\frac{{\rm d}({\rm L}_A)}{1+v}.
\end{align*}
\end{proof}

\section{Primitive sets}
Given $v\in(0,1)$, we shall be interested in elements $a\in A$ for which $P(a)^{1+v}>a$, and their multiples $ac$, where $c\in C_a^v$ for
\begin{align}\label{eq:defCa}
C_a^v \ := \ \big\{c\in\N \;: \;  p\mid c \implies p\in [P(a^*), P(a^*)^{1/\sqrt{v}})\big\}.
\end{align}
Note $c=1\in C_a^v$. Recall $a^*=a/P(a)$, so $P(a^*)$ is the second largest prime of $a$. Also if $1<c\in C_a^v$ then $P(c)\le P(a^*)^{1/\sqrt{v}}$ is markedly smaller than $P(a) \ge P(a^*)^{1/v}$.

The following key lemma provides an upgrade to Corollary \ref{cor:Lsplit} in the case when $A$ is primitive, not just ${\rm L}$-primitive. Namely, the ${\rm L}_{ac}$ are disjoint, and so the larger set $\{ac : a\in A, c\in C_a^v\}$ is ${\rm L}$-primitive.
\begin{lemma}\label{lem:disjointLac}
Let $A$ be a primitive set of composite numbers, and take $v\in (0,1)$. If $P(a)^{1+v}>a$ for all $a\in A$, then the collection of sets ${\rm L}_{ac}$, ranging over $a\in A, c\in C_a^v$, are pairwise disjoint.
\end{lemma}
\begin{proof}
Suppose ${\rm L}_{ac}\cap {\rm L}_{a'c'}\neq\emptyset$ for some $a, a'\in A$ and $c\in C_a^v$, $c'\in C_{a'}^v$. Without loss, by Lemma \ref{lem:trichot} we may assume $ac\in {\rm L}_{a'c'}$. Note if $c=1$ then $a\in {\rm L}_{a'c'}$ implies $a'\mid a'c'\mid a$, which forces $a=a'$ and $c'=1$ by primitivity of $A$. So assuming $(a,c)\neq (a',c')$ we deduce $c>1$. 

We factor $ac=p_1\cdots p_k$ into primes $p_1\ge\cdots\ge p_k$, so $ac\in {\rm L}_{a'c'}$ implies $a'c'=p_j\cdots p_k$ for some index $1<j<k$. Since $P(a)>P(c),p(c)\ge P(a^*)$ we also have $a^* = p_i\cdots p_k$ for some $2<i\le k$. If $i\le j$ then $a'c'\mid a^*$ so $a'\mid a$, contradicting $A$ as primitive. Hence $i>j$ so $a^*\mid a'c'$. Write $da^*=a'c'$ where $d=p_j\cdots p_{i-1}$, and note $P(d)=p_j=P(a')$. By definition of $1<c\in C_a^v$, we have
\begin{align}\label{eq:qPastar1}
p_j = P(d) \le P(c)<  P(a^*)^{1/\sqrt{v}}.
\end{align}

Recall $P(a')^{v}> (a')^* \ge P((a')^*)$ for $a'\in A$. Now consider cases $c'>1$ and $c'=1$. When $1<c'\in C_{a'}^v$, we have $P(c')=p_{j+1}\ge p_i= P(a^*)$. Thus
\begin{align}\label{eq:qPastar2}
p_j = P(a') > P((a')^*)^{1/v} > P(c')^{1/\sqrt{v}} \ge P(a^*)^{1/\sqrt{v}}.
\end{align}
But \eqref{eq:qPastar2} contradicts \eqref{eq:qPastar1}, so ${\rm L}_{ac}$ and ${\rm L}_{a'c'}$ are disjoint. 

Similarly when $c'=1$, we have $P((a')^*) = p_{j+1} \ge p_i = P(a^*)$ and so
\begin{align*}
p_j = P(a') > P((a')^*)^{1/v} \ge P(a^*)^{1/v}.
\end{align*}
This also contradicts \eqref{eq:qPastar1} (indeed $v<\sqrt{v}$). Hence ${\rm L}_{ac}$ and ${\rm L}_{a'}$ are disjoint in both cases.
\end{proof}

\begin{remark}
The exponent $1/\sqrt{v}$ in the definition of $C_a^v$ in \eqref{eq:defCa} is chosen as large as possible, constrained by the final steps \eqref{eq:qPastar1}, \eqref{eq:qPastar2} above. If one established a larger exponent in Lemma \ref{lem:disjointLac}, this would improve the final savings factor $\int_0^1\frac{\dd}{\dd v}[v^{1/2}]\frac{\dd{v}}{1+v} = \pi/4$. 
\end{remark}

In the following proposition, we use Lemma \ref{lem:disjointLac} in order to bound the density of ${\rm L}_{A_n}$ by essentially a savings factor $\sqrt{v}$ from the trivial bound ${\rm d}({\rm L}_{n})$, when $P(a)^{1+v}>a$ for all $a\in A_n$.

\begin{proposition}\label{prop:maindensity}
Let $A$ be a finite primitive set. Take $v\in (0,1)$, an integer $n>1$ with $n\notin A$, and denote $q=P(n)$. If $P(a)^{1+v}>a$ for all $a\in A_n$ then
\begin{align}\label{eq:maindensity}
{\rm d}({\rm L}_{A_n}) \ \le \ \sqrt{v}\; r_q\;{\rm d}({\rm L}_n)
\end{align}
for the ratio $r_q := M_q/m_q$ when $q\ge3$, and $r_2:=r_3$.
\end{proposition}
\begin{proof}
Without loss assume $A = A_n$. Then $ac\in {\rm L}_n$ for all $a\in A$, $c\in C_a^v$ (recall $p(ac)=p(a)$), and so ${\rm L}_{ac} \subset {\rm L}_n$. Note the condition $P(a)^{1+v} > a$ is equivalent to $P(a)^v > a^*$, and $v<1$ implies $P(a)\nmid a^*$. By Lemma \ref{lem:disjointLac}, we have the following (finite) disjoint union,
\begin{align}\label{eq:disjointunion}
{\rm L}_n \ \supset \ \bigcup_{a\in A}\bigcup_{c\in C_a^v}{\rm L}_{ac}.
\end{align}
Thus taking the density of \eqref{eq:disjointunion}, we obtain
\begin{align}\label{eq:denprop}
{\rm d}({\rm L}_n) & \ge {\rm d}\bigg(\bigcup_{a\in A} \bigcup_{c\in C_a^v}{\rm L}_{ac}\bigg) 
= \sum_{a\in A}\sum_{c\in C_a^v}{\rm d}({\rm L}_{ac}) = \sum_{a\in A}{\rm d}({\rm L}_a)\,\sum_{c\in C_a}\frac{1}{c},
\end{align}
noting $P(a)>P(c)$ for $1<c\in C_a^v$, so ${\rm L}_{ac} = \{bac: p(b) \ge P(a)\} = c\cdot{\rm L}_{a}$. Then by definitions of $C_a^v$ and $\mu_q$ in \eqref{eq:defCa} and \eqref{eq:muxdef},
\begin{align}\label{eq:sumCa}
\sum_{c\in C_a^v}\frac{1}{c}\  & = \prod_{p\in [P(a^*), P(a^*)^{1/\sqrt{v}})}\Big(1-\frac{1}{p}\Big)^{-1} = \prod_{p<P(a^*)^{1/\sqrt{v}}}\Big(1-\frac{1}{p}\Big)^{-1}\prod_{p<P(a^*)}\Big(1-\frac{1}{p}\Big) \nonumber\\
&= \frac{\log P(a^*)^{1/\sqrt{v}}}{\mu_{P(a^*)^{1/\sqrt{v}}}} \frac{\mu_{P(a^*)}}{\log P(a^*)} = \frac{\mu_{P(a^*)}}{\mu_{P(a^*)^{1/\sqrt{v}}}}\frac{1}{\sqrt{v}}.
\end{align}

When $q\ge3$, we use $\mu_{P(a^*)}/\mu_{P(a^*)^{1/\sqrt{v}}} \ge m_q/M_q=1/r_q$, which follows by monotonicity of $m_q,M_q$ in Lemma \ref{lem:monotonic}, and that $P(a^*),q\in\mathcal P$. Hence plugging \eqref{eq:sumCa} back into \eqref{eq:denprop}, 
\begin{align*}
{\rm d}({\rm L}_n) & \ge \frac{1}{\sqrt{v}\,r_q}\sum_{a\in A}{\rm d}({\rm L}_a) = \frac{1}{\sqrt{v}\,r_q}\,{\rm d}({\rm L}_A)
\end{align*}
as desired.

The result similarly holds when $q=2$: if $P(a^*)\ge 3$ then $\mu_{P(a^*)}/\mu_{P(a^*)^{1/\sqrt{v}}} \ge m_3/M_3=1/r_3$ as before. And if $P(a^*)= 2$ then $\mu_{2}/\mu_{2^{1/\sqrt{v}}} \ge 1$ also suffices.
\end{proof}


\section{Deduction of Theorems \ref{thm:EPS}, \ref{thm:Estrongodd}, \ref{thm:ESS}}

We now apply our analysis of the density of L-multiples to our original sum of interest $f(A) = \sum_{a\in A}\frac{1}{a\log a}$. First we need a simple lemma on bounding certain monotonic sequences.

\begin{lemma}\label{lem:mass}
For $k\ge1$, let $c_0\ge c_1\ge\cdots\ge c_k\ge 0$ and $0=D_0\le D_1\le \cdots \le D_k$. If $d_1,\ldots,d_k\ge0$ satisfy $\sum_{j\le i}d_j \le D_i$ for all $i\le k$, then we have
\begin{align*}
\sum_{i\le k}c_i d_i \ \le \ \sum_{i\le k}c_i(D_i - D_{i-1}).
\end{align*}
\end{lemma}
\begin{proof}
By rearranging sums,
\begin{align*}
\sum_{i\le k}c_id_i=\sum_{i\le k}c_i\Big(\sum_{j\le i}d_j-\sum_{j\le i-1}d_j\Big) =\sum_{i\le k-1}(c_i-c_{i+1})\sum_{j\le i}d_j \ + \ c_k\sum_{i\le k}d_i.
\end{align*}
Since $c_i \ge c_{i+1}$ and $\sum_{j\le i} d_j \le D_i$, we conclude
\begin{align*}
\sum_{i\le k}c_id_i \le \sum_{i\le k-1}(c_i-c_{i+1})D_i \ + \ c_kD_k = \sum_{i\le k}c_i (D_i-D_{i-1}).
\end{align*}
\end{proof}

To motivate the remainder of the proof, we offer a probabilistic interpretation of Proposition \ref{prop:maindensity}: for $v\ge0$, consider $D(v):=\sup_A{\rm d}({\rm L}_{A_n})/{\rm d}({\rm L}_n)$, ranging over primitive sets $A$ such that $P(a)^{1+v} > a$ for all $a\in A$. Note $D(v)$ may be viewed as a `cumulative distribution function', since $D(0)=0$ and $D(v)\to 1$ as $v\to\infty$. Now Proposition \ref{prop:maindensity} essentially bounds $D(v)$ by $\sqrt{v}$. Using the corresponding bound $1/2\sqrt{v}$ for the `probability density function', we establish quantitative bounds below.

\begin{proposition}\label{prop:oddstrongdLn}
For any primitive set $A$, and any integer $n\notin A$ with $q=P(n)\ge3$,
\begin{align*}
f(A_n) \ \le \ \frac{\pi}{4}\frac{M_q}{m_q^2}\;e^\gamma {\rm d}({\rm L}_n).
\end{align*}
\end{proposition}
\begin{proof}
Without loss, we may assume $A=A_n$ is finite, since $f(A) = \lim_{x\to\infty} f(A\cap[1,x])$. Also $n\notin A$ implies all elements of $A$ are composite.

Take $k\ge1$ and any sequence $0=v_0<v_1<\cdots<v_k=1$ and partition the set $A = \bigcup_{0\le i\le k} A_{(i)}$, where $A_{(k)}=\{a\in A: P(a)^2 \le a\}$ and for $0\le i\le k$,
\begin{align*}
A_{(i)} = \{a\in A :  P(a)^{1+v_{i}} \le a < P(a)^{1+v_{i+1}}\}.
\end{align*}
Then applying Lemma \ref{lem:fMertPnu} to each $A_{(i)}$,
\begin{align}\label{eq:fAdLAweight}
f(A) & = \sum_{0\le i\le k} f(A_{(i)}) \ \le \ \frac{e^\gamma}{m_q} \sum_{0\le i\le k} \frac{{\rm d}({\rm L}_{A_{(i)}})}{1+v_i}.
\end{align}
Note since $A$ is primitive, $\{{\rm L}_{A_{(i)}}\}_{i\le k}$ are pairwise disjoint. Also for each $j< k$, the first $j$ components are $\bigcup_{0\le i\le j}A_{(i)}=\{a\in A:a<P(a)^{1+v_{j+1}}\}=:A^{(j)}$, so by Proposition \ref{prop:maindensity} they have density
\begin{align*}
\sum_{0\le i\le j}{\rm d}({\rm L}_{A_{(i)}}) = {\rm d}({\rm L}_{A^{(j)}})\le \sqrt{v_{j+1}} \,r_q\,{\rm d}({\rm L}_n).
\end{align*}
Also for $j=k$ we have $\sum_{0\le i\le k}{\rm d}({\rm L}_{A_{(i)}}) = {\rm d}({\rm L}_{A}) \le {\rm d}({\rm L}_n)$, which is trivially less than $r_q{\rm d}({\rm L}_n)$. Let $c_i = \frac{1}{1+v_{i}}$, $d_i = {\rm d}({\rm L}_{A_{(i)}})$, $D_i = \sqrt{v_{i+1}} \,r_q\,{\rm d}({\rm L}_n)$ (here we let $v_{k+1}=v_k$ so that $D_k - D_{k-1}=0$). Thus by Lemma \ref{lem:mass} we have
\begin{align*}
\sum_{0\le i\le k} \frac{{\rm d}({\rm L}_{A_{(i)}})}{1+v_{i}} = \sum_{0\le i\le k}c_id_i \ \le \ \sum_{0\le i\le k}c_i(D_i - D_{i-1})
=r_q\,{\rm d}({\rm L}_n)\sum_{0\le i\le k} \frac{\sqrt{v_{i+1}}-\sqrt{v_{i}}}{1+v_{i}}.
\end{align*}
Hence the weighted sum in \eqref{eq:fAdLAweight} is bounded by
\begin{align}\label{eq:Riemannsum}
f(A) & \le \frac{r_q}{m_q}\,e^\gamma\,{\rm d}({\rm L}_n)\sum_{1\le i\le k} \frac{\sqrt{v_{i}}-\sqrt{v_{i-1}}}{1+v_{i-1}}.
\end{align}
As \eqref{eq:Riemannsum} holds for any partition $0=v_0<v_1<\cdots<v_k=1$, we may set $v_i=\frac{i}{k}$ and obtain the corresponding integral,
\begin{align*}
\lim_{k\to\infty}\sum_{1\le i\le k} \frac{\sqrt{v_i}-\sqrt{v_{i-1}}}{1+v_{i-1}} = \lim_{k\to\infty}\sum_{1\le i\le k}\int_{v_{i-1}}^{v_i} \frac{\dd}{\dd v}\Big[\sqrt{v}\Big]\frac{\dd{v}}{1+v_{i-1}} 
= \int_0^1 \frac{\dd{v}}{2\sqrt{v}(1+v)} = \frac{\pi}{4}.
\end{align*}
Hence we conclude $f(A) \le \frac{\pi}{4}\,\frac{r_q}{m_q}\,e^\gamma\,{\rm d}({\rm L}_n)$ as desired.
\end{proof}

We illustrate the value of these bounds by deducing Theorem \ref{thm:Estrongodd} in quantitative form.
\begin{corollary}
Let $A$ be a primitive set, and take an odd prime $p$. If $p\notin A$ then we have $f(A_p) <.901\,f(p)$, and moreover $f(A_p) \le (\frac{\pi}{4}+o(1))f(p)$ as $p\to\infty$. In addition, if $p>23$ and $2p\notin A$ then $f(A_{2p}) < f(2p)$.
\end{corollary}
\begin{proof}
For an odd prime $q$ define $b_q:=\frac{\pi}{4}\frac{M_q}{m_q^2}\mu_q$. Then Proposition \ref{prop:oddstrongdLn} shows that if $n\notin A$ we have
\begin{align*}
f(A_n) \le \frac{\pi}{4}\frac{M_q}{m_q^2}e^\gamma {\rm d}({\rm L}_n) = \frac{q}{n}\,b_q f(q)
\end{align*}
with $q=P(n)\ge3$, recalling ${\rm d}({\rm L}_n) = \frac{q}{n}{\rm d}({\rm L}_q)$ and \eqref{eq:fqdLq}. In particular for $n=q,2q$ we have $f(A_q)\le b_qf(q)$ and $f(A_{2q})\le \frac{1}{2}b_qf(q)$. Note $\mu_q,m_q,M_q\sim 1$ implies $b_q \sim\frac{\pi}{4}$ as claimed. Also the first few values of $b_q$ are displayed below.
\[\begin{array}{cc|cc}
q & b_q & q & b_q\\
\hline
3  &  0.9006 & 23 &  0.8232\\
5  &  0.8795 & 29 &  0.8266\\
7  &  0.8507 & 31 &  0.8139\\ 
11 &  0.8564 & 37 &  0.8184\\ 
13 &  0.8327 & 41 &  0.8189\\ 
17 &  0.8491 & 43 &  0.8092\\
19 &  0.8305 & 47 &  0.8090
\end{array}\]

Observe for $q>7$, we have
\begin{align*}
f(A_q) \le \frac{\pi}{4} \Big(\frac{M_q}{m_{11}}\Big)^2 f(q) \le \frac{\pi}{4} \Big(\frac{1+1/2\log(2\cdot 10^9)^2}{\mu_{19}}\Big)^2 f(q) < .879 f(q).
\end{align*}
In particular, with the table, we see $f(A_q) < .901 f(q)$ for all $q>2$ as claimed.

Finally, we note $f(A_{2q}) < f(2q)$ whenever $b_q < \frac{\log q}{\log(2q)}$. The result then follows since
\begin{align}\label{eq:logp2p}
b_q = \frac{\pi}{4}\frac{M_q}{m_q^2} \mu_q \ > \ \frac{\log q}{\log(2q)}  \quad\qquad\text{iff} \quad q\le 23.
\end{align}
Indeed, this may be checked directly for $q< 47$. And for $q\ge47$ we observe that $\log q/\log(2q)\ge\log47/\log94\ge.847$ exceeds $b_q \le \frac{\pi}{4}(M_{2\cdot 10^9}/m_{47})^2 \le .834$.
\end{proof}

Importantly $b_q<1$ for all odd $q$, which means every odd prime is Erd\H os strong. However it remains an open question whether $q=2$ is Erd\H os strong. Now if $2\in A$ we immediately deduce $f(A) \le f(\mathcal P)=1.6366\cdots$. Thus to complete the proof of Theorem \ref{thm:EPS}, it suffices to assume $2\notin A$. 

We achieve this in the result below. The argument is somewhat similar in spirit to that of Theorem 1.1 and Lemma 2.4 in \cite{LPprim}.

\begin{theorem}\label{thm:evenPS}
For any primitive set $A$ with $2\notin A$, we have $f(A) < 1.60$.
\end{theorem}
\begin{proof}
As $2\notin A$, denote by $K\ge 2$ the exponent for which $2^K \in A$. Note $K$ is unique by primitivity (Also if $2^k\notin A$ for all $k$ let $K=\infty$, in which case let $f(2^K)=0$). Partition $A$ into sets $A^0=\{a\in A : 2\nmid a\}$ and $A^k = \{a\in A: 2^k\| a\}$ for $k\ge1$, and let $B^k = \{a/2^{k}: a\in A^k\}$. We have
\begin{align}\label{eq:fAsplit2kp}
f(A) & = f(2^K)+\sum_{p\in A}f(p) +\sum_{p\notin A}f(A_{p}) \nonumber\\
&\le f(2^K) + \sum_{\substack{p>2\\p\in A}} f(p) + \sum_{\substack{p>2\\p\notin A}} b_pf(p) +  \sum_{\substack{p>2\\p\notin A}}\sum_{k=1}^{K-1} f((A^k)_{2^k p}),
\end{align}
since $f((A^0)_p) \le b_pf(p)$ if $p\notin A$ by Proposition \ref{prop:oddstrongdLn}. More generally, if $2^kp\notin A$ then
\begin{align*}
f((A^k)_{2^k p}) \le 2^{-k}f((B^k)_p) \le 2^{-k} b_p f(p).
\end{align*}
By comparison if $2^kp\in A$ then $f((A^k)_{2^k p}) = f(2^kp) \le 2^{-k}f(p)\frac{\log p}{\log(2p)}$.

Observe that either $2^kp\notin A$ for all $k\ge1$, or $2^Jp\in A$ for a (unique) $J=J_p\in [1,K)$, in which case $(A^k)_{2^k p}=\emptyset$ for all $k>J$ by primitivity. Thus by \eqref{eq:logp2p}, it suffices to assume $2^kp\notin A$ for all $k\ge1$ when $p\le 23$, and $2^Jp\in A$ for some $J\in [1,K)$ when $p> 23$, so
\begin{align*}
\sum_{\substack{p>2\\p\notin A}}\sum_{k=1}^{K-1} f((A^k)_{2^k p}) & \le  (1-2^{1-K})\sum_{\substack{2<p\le 23\\p\notin A}} b_pf(p)  + \sum_{\substack{p> 23\\p\notin A}}f(p)\Big((1-2^{1-J})b_p + 2^{-J}\frac{\log p}{\log(2p)} \Big)\\
& \le  (1-2^{1-K})\sum_{\substack{2<p\le 23\\p\notin A}} b_pf(p) + \sum_{\substack{p> 23\\p\notin A}}b_p f(p),
\end{align*}
since $2b_p > 1 > \frac{\log p}{\log(2p)}$ for all $p>2$. Moreover $(2-2^{1-K})b_p \ge (2-1/2)\frac{\pi}{4} > 1.1$, so \eqref{eq:fAsplit2kp} becomes
\begin{align}\label{eq:fAC12}
f(A)
& \le f(2^K) + \sum_{\substack{p>2\\p\in A}} f(p) + (2-2^{1-K})\sum_{\substack{2<p\le 23\\p\notin A}} b_pf(p) + 2\sum_{\substack{p> 23\\p\notin A}}b_p f(p) \nonumber\\
& \le f(2^K) + (2-2^{1-K})\sum_{2<p\le 23} b_pf(p) + 2\sum_{p> 23}b_p f(p)\nonumber\\
& =: \ f(2^K) + (2-2^{1-K})C_1 + 2C_2.
\end{align}

Now we compute the constants $C_1,C_2$. First, let $M=M_{2\cdot10^9}=1.001\cdots$. Recalling $\mu_p f(p) = e^\gamma {\rm d}({\rm L}_p)$,
\begin{align}\label{eq:C2}
C_2 :=\sum_{p>23}b_p f(p) = \frac{\pi}{4} e^\gamma \sum_{p> 23}\frac{M_p}{m_p^2}{\rm d}({\rm L}_p) \le \frac{\pi}{4}\frac{Me^\gamma}{\mu_{23}^2} \prod_{p\le 23}\big(1-\tfrac{1}{p}\big) = 0.251135\cdots,
\end{align}
since $\sum_{p> q}{\rm d}({\rm L}_{p}) = \prod_{p\le q}(1-\frac{1}{p})$. Similarly we have
\begin{align}\label{eq:C1}
C_1 :=\sum_{2<p\le 23}b_p f(p)= \sum_{2< p\le 23} \frac{\pi}{4}\frac{M}{m_p^2} e^\gamma{\rm d}({\rm L}_p) = \frac{\pi}{4}\,M e^\gamma\cdot 0.39012\cdots = 0.5463\cdots
\end{align}
Here we computed
\begin{align*}
\sum_{2< p\le 23} \frac{1}{m_p^2} {\rm d}({\rm L}_{p}) & = \frac{1}{\mu_7^2}\sum_{2< p\le 7}{\rm d}({\rm L}_{p}) + \frac{1}{\mu_{19}^2}\sum_{7< p\le 19}{\rm d}({\rm L}_{p}) + \frac{1}{\mu_{23}^2} {\rm d}({\rm L}_{23}) 
= 0.390126\cdots,
\end{align*}
using $\sum_{q< p\le q'}{\rm d}({\rm L}_{p}) = \prod_{p\le q}(1-\frac{1}{p})-\prod_{p\le q'}(1-\frac{1}{p})$.

Hence plugging \eqref{eq:C2} and \eqref{eq:C1} back into \eqref{eq:fAC12},
\begin{align}
f(A) & \le f(2^K) + (2-2^{1-K})C_1 + 2C_2 \nonumber\\
& \le 2^{-K}\Big(\frac{1}{\log 4} - 2C_1\Big) + 2(C_1 + C_2)
\le 2(C_1 + C_2) \le 1.595.
\end{align}
Here we used $2C_1 > .722 > 1/\log 4$. This completes the proof.
\end{proof}
\begin{remark}
A similar argument as in Theorem \ref{thm:evenPS} shows $f(A_2) < C_1+C_2 < 0.80$ when $2\notin A$. We leave this to the interested reader. Note this bound improves on $f(A_2) < e^\gamma/2\approx 0.89$ from \cite[Proposition 2.1]{LPprim}, but unfortunately still exceeds $f(2) \approx 0.72$.
\end{remark}

\subsection{Proof of Theorem \ref{thm:ESS}}
Take $\eps>0$. We shall introduce large parameters $y=y_\eps$, $k=k_{\eps,y}$, and $x=x_{\eps,k}$.

By Lemma \ref{lem:LMertfA}, we have $f(A_n) \le e^\gamma{\rm d}({\rm L}_n)$ for any integer $n\notin A$, $n>1$, and when $y=y_\eps\in\R$ is sufficiently large by Proposition \ref{prop:oddstrongdLn} we have the sharper bound
\begin{align}\label{eq:AnPny}
f(A_n) \ \le \ (\frac{\pi}{4}e^\gamma + \eps){\rm d}({\rm L}_n) \qquad\text{provided}\quad P(n) > y.
\end{align}

Next by \cite[Lemma 2]{ErdSark}, for $k=k_{\eps}=k_{\eps,y}\in\N$ sufficiently large,
\begin{align}\label{eq:NkPnley}
\sum_{\substack{n\in \N_{k}\\P(n)\le y}}{\rm d}({\rm L}_n) \ll \frac{1}{k}\sum_{\substack{\Omega(n)=k\\P(n)\le y}}\frac{1}{n}  \ll (\log y)^2\,2^{-k} < \eps.
\end{align}

Finally, since $f(\N_j)<2$ crudely for all $j$ there exists $x=x_{\eps,k}\in\R$ sufficiently large so that $f\big(\bigcup_{j \le k}\N_j\cap [x,\infty)\big)<\eps$.

Now take a primitive set $A\subset [x,\infty)$, and consider the partition $A = A'\cup \bigcup_{n\in\N_k \setminus A}A_n$, where $A'$ consists of elements $a\in A$ with at most $k$ prime factors, and each other element $a\in A$ (with at least $k+1$ prime factors) then lies in $A_n=A\cap {\rm L}_n$, where $n\notin A$ is the product of the smallest $k$ primes of $a$. Hence we conclude
\begin{align*}
f(A) & = f(A') + \sum_{\substack{n\in \N_{k}\setminus A}}f(A_n)\\
& \le f\big(\bigcup_{j \le k}\N_j\cap [x,\infty)\big) + \sum_{\substack{n\in \N_{k}\setminus A\\P(n)\le y}}f(A_n)+ \sum_{\substack{n\in \N_{k}\setminus A\\P(n)> y}}f(A_n)\\
& \le \eps + e^\gamma\sum_{\substack{n\in \N_{k}\\P(n)\le y}}{\rm d}({\rm L}_n)+ (\frac{\pi}{4}e^\gamma + \eps)\sum_{\substack{n\in \N_{k}\\P(n)>y}}{\rm d}({\rm L}_n) \ \le \ \eps + e^\gamma\,\eps+ (\frac{\pi}{4}e^\gamma + \eps),
\end{align*}
by \eqref{eq:AnPny}, \eqref{eq:NkPnley}, and noting $\sum_{n\in \N_{k}, P(n)>y}{\rm d}({\rm L}_n) \le 1$. Hence letting $\eps\to 0$ completes the proof of Theorem \ref{thm:ESS}.

\section{{\rm L}-primitive sets revisited}

\subsection{Upper density} As mentioned in the introduction, one of the striking early results in the study of primitive sets was due to Besicovitch \cite{Besicovitch}, who showed
\begin{align*}
\sup_{A\textnormal{ primitive}}\overline{{\rm d}}(A) \;=\; \frac{1}{2}.
\end{align*} 
This came as quite a surprise at the time, in particular disproving a conjecture of Davenport. We shall extend this phenomenon further to {\rm L}-primitive sets, in Proposition \ref{prop:LBesicovitch}.

To proceed, we recall a result of Erd\H{o}s \cite{ErdosBes}, which bounds the density of the set of multiples of an interval. Also see Hall--Tenenbaum \cite[Theorem 21]{HTdivisor} for quantitatively stronger results. Denote the set of (all) multiples of $A\subset \N$ as ${\rm M}_A = \{na : n\in\N, a\in A\}$.

\begin{proposition}[Erd\H{o}s, 1936] \label{prop:ErdBes}
Let $\varepsilon(x)$ be any function with $\varepsilon(x)\to 0$ as $x\to\infty$. Then the upper density of ${\rm M}_{(x^{1-\varepsilon(x)},x]}$ tends to zero as $x\to\infty$.
\end{proposition}

We prove a Besicovitch-type result for {\rm L}-primitive sets, notably with full upper density.

\begin{proposition}\label{prop:LBesicovitch}
We have $\sup_A \overline{\rm d}(A) = 1$ over {\rm L}-primitive sets $A$.
\end{proposition}
\begin{proof}
Take $h\in\Z_{>1}$, $\eps>0$ and let $S = \{n\in\N : P(n)\le h\}$ be the set of $h$-smooth numbers. For a sequence of indices $k_1,k_2,\ldots$ to be determined, define intervals $I_i = (h^{k_i-1}, h^{k_i}]$. Let $S_i:=I_i\setminus S$ and note for $a\in S_i$ and $n\in {\rm L}_a$ we have $n\ge P(a)a >h^{k_i}$, so $n\notin I_i\supset S_i$. In particular $a'\notin {\rm L}_a$ for distinct $a,a'\in S_i$, so each set $S_i$ is {\rm L}-primitive. Now define the {\rm L}-primitive set
\begin{align}
A = \bigcup_{j\ge1}\bigg(S_j\setminus\bigcup_{1\le i<j}{\rm M}_{I_i}\bigg).
\end{align}
Note for each fixed $h>1$ the set $S$ has zero density, so $|S\cap [1,x]| < \epsilon x$ for $x\ge x_{h,\epsilon}$ sufficiently large. Also by Proposition \ref{prop:ErdBes} we see $\overline{\rm d}({\rm M}_{(x/h,x]})\to0$ as $x\to\infty$. So for $k_i$ large enough, we may assume $\overline{\rm d}({\rm M}_{I_i}) < \eps/2^i$. 

For each $i$ the set of multiples ${\rm M}_{I_{i}}$ is a periodic set with period (dividing) $(h^{k_i})!$. So assuming 
$k_{i+1} \ge (h^{k_{i}})!$ the relative density of ${\rm M}_{I_i}$ inside $I_{i+1}$ is at most $2\overline{\rm d}({\rm M}_{I_i})$. Hence
\begin{align*}
|A\cap [1,h^{k_j}]| &\ge |I_j| - \big|S\cap [1,h^{k_j}]\big| - 2h^{k_j}\sum_{1\le i<j}\overline{\rm d}({\rm M}_{I_i})\\
&\ge (h^{k_j}-h^{k_j-1}) - \epsilon h^{k_j} - 2\eps h^{k_j}\sum_{i\ge1}2^{-i}.
\end{align*}
Thus dividing by $x=h^{k_j}$ we see $\overline{\rm d}(A)=\limsup_{x\to\infty}|A\cap [1,x]|/x \ge 1 - 1/h - 3\eps$. Taking $h\to\infty$ and $\eps\to0$ completes the proof.
\end{proof}

\subsection{The Erd\H{o}s {\rm L}-primitive set conjecture}

Sets of {\rm L}-multiples play a central role in our proof of Theorem \ref{thm:EPS}, as the mathematical structures arising from a probabilistic interpretation of \eqref{eq:sumofden},\footnote{A variant in \cite{ESS67}, a set $A$ `possesses property I' if there is no solution to $a'=ba$ for $a,a'\in A$ with $p(b)>P(a)$. This is similar to $A$ as {\rm L}-primitive, but the latter imposes the inclusive inequality $p(b)\ge P(a)$, which arises naturally from a probabilistic viewpoint. This inclusivity leads to key structural properties, notably the trichotomy in Lemma \ref{lem:trichot}.} and implicit in the original 1935 argument of Erd\H{o}s \cite{Erdos35}.\footnote{The author was also recently shown `prefix-free sets' in \cite{AKS}, which coincides with {\rm L}-primitive for sets of squarefree numbers.}
As such it is natural to pose the {\rm L}-primitive analogue of Conjecture \ref{conj:EPS}, namely that $f(A) \le f(\mathcal P)$ for all {\rm L}-primitive sets $A$. 

However, this conjecture turns out to be false.

\begin{proposition}\label{prop:Lprimsup}
We have
\begin{align}\label{eq:Lprimsupmax}
\sup_{A\textnormal{ L-primitive}} f(A) \ &= \ \sum_p\max\{f(p),e^\gamma{\rm d}({\rm L}_p)\},
\quad \text{and} \qquad
\lim_{x\to\infty}\sup_{\substack{A\subset [x,\infty)\\A\textnormal{ L-primitive}}} f(A) \ = \ e^\gamma.
\end{align}
\end{proposition}
Note the prime sum in \eqref{eq:Lprimsupmax} above is at least (and well-approximated by) $f(\mathcal P)-f(2) + e^\gamma/2 \approx 1.805$. In particular it exceeds $f(\mathcal P) \approx 1.636$. As such Proposition \ref{prop:Lprimsup}, along with Conjecture \ref{conj:BM} and related work in the literature, highlights how the Erd\H{o}s primitive set conjecture is quite fragile under certain seemingly natural directions of generalization.

We now proceed to set up the proof of Proposition \ref{prop:Lprimsup}. First, the trichotomy in Lemma \ref{lem:trichot} leads to the following.
\begin{lemma}\label{lem:genset}
Every set $S\subset \N$ has a unique {\rm L}-primitive subset $\langle S\rangle$ with ${\rm L}_{\langle S\rangle} = {\rm L}_S$. In particular $\langle S\rangle=S$ if $S$ is {\rm L}-primitive.
\end{lemma}
\begin{proof}
For any $s_1,s_2\in S$, by Lemma \ref{lem:trichot} either ${\rm L}_{s_1}\cap {\rm L}_{s_2}=\emptyset$ or ${\rm L}_{s_1}\subset {\rm L}_{s_2}$ (or vice versa). Thus each $s\in S$ has a (unique) smallest {\rm L}-divisor $s'\in S$, inducing a map $S\to S:s\mapsto s'$. We define $\langle S\rangle$ as the image of this map. Explicitly this is
\begin{align}
\langle S\rangle := \{s\in S : s\notin {\rm L}_t \;\forall\; t<s, t\in S\}.
\end{align}
By minimality ${\rm L}_{s_1}\cap {\rm L}_{s_2}=\emptyset$ for all $s_1,s_2\in \langle S\rangle$, so $\langle S\rangle$ is {\rm L}-primitive. Moreover ${\rm L}_S = \bigcup_{s\in S}{\rm L}_s = \bigcup_{s'\in \langle S\rangle}{\rm L}_{s'} = {\rm L}_{\langle S\rangle}$, where the latter union over $\langle S\rangle$ is disjoint by {\rm L}-primitivity. This completes the proof.
\end{proof}

Next take $v>0$, $n\in\Z_{>1}$, and consider the set $D_v(n)$ of prime divisors of $n$ whose induced {\rm L}-divisor is not smooth, i.e.
\begin{align}
D_v(n) = \Big\{ p\mid n \; :\, \prod_{q^e\| n, q< p} q^e \ \le \ p^v \Big\}.
\end{align}
We cite the following result of Bovey, based on earlier work of Erd\H{o}s \cite[\S1.2]{HTdivisor}.
\begin{proposition}[Bovey, 1977] \label{prop:Bovey}
For each $v>0$ there is a set $N_v\subset \N$ of full density with
\begin{align}\label{eq:Dickman}
\frac{|D_v(n)|}{\log\log n} \ \to \ e^{-\gamma}\int_0^{v} \rho(x)\dd{x}
\end{align}
as $n\to\infty$ on $N_v$. Here $\rho$ is the Dickman--de Bruijn function.
\end{proposition}
\begin{remark}
In probability, the righthand side of \eqref{eq:Dickman} is called the Dickman distribution.
\end{remark}

In particular, $|D_v(n)|\gg_v \log\log n$ for all $n\in N_v$. Now we may define a map $\beta: N_u\to \N$ sending $n$ to its {\rm L}-divisor $\beta(n)=p\prod_{q^e\| n, q< p} q^e$, for the largest prime $p\in D_v(n)$.

Define the {\rm L}-primitive generating set $B(v) := \langle \beta(N_v)\rangle$ as in Lemma \ref{lem:genset}. By construction ${\rm L}_{B(v)} = {\rm L}_{\beta(N_v)} \supset N_u$ has full density. Also, by definition of $\beta, D_v$,
\begin{align}\label{eq:P1u}
B(v)\subset \beta(N_v) \subset \{n\in\N : n \le P(n)^{1+v}\}.
\end{align}

We are now prepared to establish a local version of Proposition \ref{prop:Lprimsup}.
\begin{proposition}\label{prop:limsupLfAq}
For each prime $q$, we have
\begin{align*}
\lim_{y\to\infty}\sup_{\substack{A\subset[y,\infty)\\\textnormal{L-primitive }A\not\ni q}}f(A_q) \ = \ \sup_{\textnormal{L-primitive }A\not\ni q}f(A_q) \ = \ e^\gamma {\rm d}({\rm L}_q).
\end{align*}
\end{proposition}
\begin{proof}
By Lemma \ref{lem:LMertfA} we have $f(A_p) < e^\gamma {\rm d}({\rm L}_p)$ for all {\rm L}-primitive $A$ not containing $p$. It now suffices to provide {\rm L}-primitive sets $B\subset[y,\infty)$ with $f(B_q) \to e^\gamma {\rm d}({\rm L}_q)$ as $y\to\infty$.

Fix $v>0$. The {\rm L}-primitive set $B(v)$ in \eqref{eq:P1u} satisfies
\begin{align}\label{eq:fBuq}
f(B(v)_q) = \sum_{b\in B(v)_q}\frac{1}{b\log b} \ & \ge \ \frac{1}{1+v}\sum_{b\in B(v)_q}\frac{1}{b\log P(b)}.
\end{align}
Next, for $x>e^{e^{e^y}}$ we may assume $N_v\subset [x,\infty)$ and retain full density. Observe then $B(v)\subset [y,\infty)$ is our candidate  {\rm L}-primitive set. Indeed, for each $n\in N_v$ by construction $\beta(n)$ is divisible by all primes $q\in D_v(n)$, so $\beta(n)$ is composite with $\beta(n)\ge |D_v(n)| \gg_v \log\log n \ge \log\log\log x > y$ for each $b\in B(v)$, for $y$ sufficiently large. And note Mertens' product theorem gives
\begin{align*}
{\rm d}({\rm L}_b) = \frac{1}{b}\prod_{p< P(b)}\Big(1-\frac{1}{p}\Big) = \frac{e^{-\gamma}+o_y(1)}{b\log P(b)}.
\end{align*}
Plugging back into \eqref{eq:fBuq}, we obtain
\begin{align*}
f(B(v)_q) \ge \frac{e^\gamma+o_y(1)}{1+v}\sum_{b\in B(v)_q}{\rm d}({\rm L}_b).
\end{align*}
Recall ${\rm L}_{B(v)} = {\rm L}_{\beta(N_v)} \supset N_v$ has full density, which implies $({\rm L}_{B(v)})_q={\rm L}_{B(v)_q}$ has full relative density ${\rm d}({\rm L}_{B(v)_q}) = {\rm d}({\rm L}_q)$. Hence by Corollary \ref{cor:Lsplit} this latter sum is
\begin{align*}
\sum_{b\in B(v)_q}{\rm d}({\rm L}_b) = {\rm d}({\rm L}_{B(v)_q}) = {\rm d}({\rm L}_q).
\end{align*}
Thus taking $y\to\infty$ and $v\to0$ gives $f(B(v)_q)\to e^\gamma{\rm d}({\rm L}_q)$ as desired. 
\end{proof}

\begin{proof}[Proof of Proposition \ref{prop:Lprimsup}]
Take {\rm L}-primitive $A\subset [x,\infty)$, so $p\ge x$ for all $p\in A$. Then by Lemma \ref{lem:LMertfA},
\begin{align*}
f(A) = \sum_p f(A_p) & = \sum_{\substack{p< x\\\text{or }p\notin A}} f(A_p) + \sum_{p\in A, p\ge x}f(p)\\
& \le e^\gamma\sum_{\substack{p< x\\\text{or }p\notin A}} {\rm d}({\rm L}_p) + \sum_{p\in A, p\ge x}f(p)\\
& \le e^\gamma\sum_{p} {\rm d}({\rm L}_p) + \big(e^\gamma+o_x(1)\big)\sum_{p\ge x}{\rm d}({\rm L}_p) \ \le \ e^\gamma+o_x(1)
\end{align*}
by Mertens' theorem, and noting $\sum_p {\rm d}({\rm L}_p)=1$. Thus $\lim_x\sup_{A\subset[x,\infty)}f(A) \le e^\gamma$. Equality in the limsup holds for the choice of $B = \bigcup_{q} B(v)_q$ and taking $v\to0$ as in Proposition \ref{prop:limsupLfAq}. Observe such $B$ inherits {\rm L}-primitivity from the $B(v)_q$. Note in general, a union $B=\bigcup_q B_q$ is {\rm L}-primitive if each $B_q$ is {\rm L}-primitive. (By contrast $B=\bigcup_q B_q$ is not necessarily primitive even if each $B_q$ is primitive, e.g. $B=\{3,6\}$.)

Next, consider the primes $\mathcal Q = \{q : f(q) > e^\gamma {\rm d}({\rm L}_q)\} = \{q : 1/\log q > e^\gamma\prod_{p<q}(1-\frac{1}{p})\}$. By Lemma \ref{lem:LMertfA} $f(A_q) < e^\gamma {\rm d}({\rm L}_q)$ when $q\notin A$, so in general $f(A_q) < \max\{f(q),e^\gamma {\rm d}({\rm L}_q)\}$ for all {\rm L}-primitive $A$ and all primes $q$. Hence
\begin{align*}
f(A) = \sum_q f(A_q) < \sum_q\max\{f(q),e^\gamma {\rm d}({\rm L}_q)\} = f(\mathcal Q) + e^\gamma\big(1-{\rm d}({\rm L}_{\mathcal Q})).
\end{align*}
This bound is attained for the choice of $B' = \mathcal Q\cup\bigcup_{q\notin  \mathcal Q} B(v)_q$, and taking $v\to0$ as in Proposition \ref{prop:limsupLfAq}. Again $B'$ inherits {\rm L}-primitivity from the $B(v)_q$, as desired.
\end{proof}

\section{Deduction of Theorems \ref{thm:LDavenErdos}, \ref{thm:ESSLchain}}

Our study of sets of {\rm L}-multiples leads to Theorem \ref{thm:LDavenErdos}, refining Davenport--Erd\H{o}s. This in turn enables the proof of Theorem \ref{thm:ESSLchain}, by a modification of the argument in \cite[Theorem 2]{ESS66}, with greater care given to the constants involved.

To proceed we first establish some lemmas.

\begin{lemma}\label{lem:densitysumdLa}
For any {\rm L}-primitive $A\subset \N$, we have $\underline{{\rm d}}({\rm L}_A)\ge \sum_{a\in A}{\rm d}({\rm L}_a)$. Moreover if $\sum_{a\in A}1/a < \infty$ then the natural density ${\rm d}({\rm L}_A)$ exists and equals $\sum_{a\in A}{\rm d}({\rm L}_a)$.
\end{lemma}
\begin{proof}
For each $a\in A$ we have $\underline{{\rm d}}({\rm L}_a) = {\rm d}({\rm L}_a)$. So taking the lower density of the finite (disjoint) union $\bigcup_{a\in A, a\le x}{\rm L}_a \subset {\rm L}_A$, we have $\sum_{a\in A, a\le x}{\rm d}({\rm L}_a) \le \underline{{\rm d}}({\rm L}_A)$ for all $x>1$. Thus $\sum_{a\in A}{\rm d}({\rm L}_a) \le \underline{{\rm d}}({\rm L}_A)$. Moreover if $\sum_{a\in A}1/a < \infty$, then for all $y>1$
\begin{align*}
\frac{1}{x}\sum_{\substack{n\le x\\ n\in {\rm L}_{A\cap (y,\infty)}}}1 \le \frac{1}{x}\sum_{a\in A,a> y}\left\lfloor \frac{x}{a}\right\rfloor \le \sum_{a\in A,a> y}\frac{1}{a} = o_y(1).
\end{align*}
Thus $\overline{{\rm d}}({\rm L}_{A\cap (y,\infty)}) \to 0$ as $y\to\infty$, and so combining with ${\rm d}({\rm L}_{A\cap [1,y]})=\sum_{a\in A, a\le y}{\rm d}({\rm L}_a)$ completes the proof.
\end{proof}

The following lemma shows that sets of {\rm L}-multiples have a $\log$ density, refining Davenport--Erd\H{o}s' elementary proof for sets of (all) multiples \cite{DE51}. 
\begin{lemma}
For any {\rm L}-primitive $A\subset \N$, the $\log$ density $\delta({\rm L}_A)$ exists and equals $\sum_{a\in A}{\rm d}({\rm L}_a)$.
\end{lemma}
\begin{proof}
In general $\underline{{\rm d}}(S) \; \le\; \underline{\delta}(S) \;\le \; \overline{\delta}(S) \; \le \; \overline{{\rm d}}(S)$ for any $S\subset \N$. So for $S={\rm L}_A$, by Lemma \ref{lem:densitysumdLa} it suffices to show
\begin{align}\label{eq:sumdLageDelt}
\sum_{a\in A}{\rm d}({\rm L}_a) \ \ge \ \overline{\delta}({\rm L}_A).
\end{align}
To this, for $y>1$ let $A^y = \{a\in A : P(a)\le y\}$ and $L^y = \{n\in {\rm L}_A : P(n)\le y\}$. Note $L^y \subset {\rm L}_{A^y}$. Also $\sum_{a\in A^y}\frac{1}{a} \le \prod_{p\le y}(1-\frac{1}{p})^{-1} = O_y(1)$, so by Lemma \ref{lem:densitysumdLa} ${\rm d}({\rm L}_{A^y})$ exists and equals $\sum_{a\in A^y}{\rm d}({\rm L}_a)$ for all $y>1$. In particular ${\rm d}({\rm L}_{A^y})\to \sum_{a\in A}{\rm d}({\rm L}_a)$ as $y\to\infty$.

Now observe each $n\in {\rm L}_{A^y}$ is a {\rm L}-multiple of a unique $a\in A^y$, so for $x\ge y>1$ we have
\begin{align}\label{eq:recipLx}
\sum_{n\in L^x\cap {\rm L}_{A^y}}\frac{1}{n} &= \sum_{a\in A^y}\frac{1}{a}\prod_{P(a)\le p\le x}(1-\tfrac{1}{p})^{-1} = \sum_{a\in A^y}\frac{1}{a}\prod_{p<P(a)}(1-\tfrac{1}{p})\prod_{p\le x}(1-\tfrac{1}{p})^{-1} \nonumber\\
&= {\rm d}({\rm L}_{A^y})\prod_{p\le x}(1-\tfrac{1}{p})^{-1}.
\end{align}
In particular for $x=y$ we have $\sum_{n\in L^x}\frac{1}{n}={\rm d}({\rm L}_{A^x})\prod_{p\le x}(1-\tfrac{1}{p})^{-1}$.

Then for all $x\ge y>1$, by \eqref{eq:recipLx} and Mertens' theorem
\begin{align}\label{eq:notrecipLy}
\sum_{n\in L^x \setminus {\rm L}_{A^y}}\frac{1}{n} 
&= \ \sum_{n\in L^x}\frac{1}{n} \ - \sum_{n\in L^x \cap {\rm L}_{A^y}}\frac{1}{n} \nonumber\\
&= \big({\rm d}({\rm L}_{A^x})-{\rm d}({\rm L}_{A^y})\big)\prod_{p\le x}(1-\tfrac{1}{p})^{-1}
\ \ll \ (\log x)\big({\rm d}({\rm L}_{A^x})-{\rm d}({\rm L}_{A^y})\big).
\end{align}
Recall the natural density ${\rm d}({\rm L}_{A^y})$ exists, in which case equals the log density $\delta({\rm L}_{A^y})$. Hence by \eqref{eq:notrecipLy}, for each $y>1$ the upper log density is
\begin{align}
\overline{\delta}({\rm L}_{A}) \ = \ \limsup_{x\to\infty}\frac{1}{\log x}\sum_{\substack{n\le x\\n\in {\rm L}_{A}}}\frac{1}{n} 
& \le \lim_{x\to\infty}\frac{1}{\log x}\sum_{\substack{n\le x\\n\in {\rm L}_{A^y}}}\frac{1}{n} \ + \ \limsup_{x\to\infty}\frac{1}{\log x}\sum_{n\in L^x \setminus {\rm L}_{A^y}}\frac{1}{n} \nonumber\\
& \ = \delta({\rm L}_{A^y}) \ + \ \lim_{x\to\infty}O\big({\rm d}({\rm L}_{A^x})-{\rm d}({\rm L}_{A^y})\big) \nonumber\\
& \ = {\rm d}({\rm L}_{A^y}) \ + \ O\Big(\sum_{a\in A}{\rm d}({\rm L}_a)-{\rm d}({\rm L}_{A^y})\Big).
\end{align}
Hence ${\rm d}({\rm L}_{A^y})\to \sum_{a\in A}{\rm d}({\rm L}_a)$ as $y\to\infty$ implies $\overline{\delta}({\rm L}_{A})\le \sum_{a\in A}{\rm d}({\rm L}_a)$, giving \eqref{eq:sumdLageDelt}.
\end{proof}

\vspace{.5em}
\noindent
{\bf Theorem \ref{thm:LDavenErdos}.} {\it If $\overline{\delta}(A)>0$, then $A$ contains an infinite {\rm L}-divisibility chain.}
\begin{proof}
We claim all such $A\subset \N$ contain an element $a\in A$ such that $A\cap {\rm L}_a$ has positive upper $\log$ density. (In other words, if $\overline{\delta}(A) > 0$ then there exists an element $a\in A$ such that $\overline{\delta}(A\cap {\rm L}_a)>0$.)

Assume this claim holds. Letting $A^1=A$, $a_1=a$, and for $i\ge1$ suppose $\overline{\delta}(A^i)>0$. By the claim there exists $a_i\in A^i$ such that $A^{i+1} := A^i\cap {\rm L}_{a_i}$ has positive upper $\log$ density. Hence by induction, we obtain an {\rm L}-divisibility chain $a_1,a_2,\cdots$, as desired.

Thus it remains to establish the above claim. For sake of contradiction, suppose $A\cap {\rm L}_a$ has zero $\log$ density for all $a\in A$. Next, for the {\rm L}-primitive generating set $B = \langle A\rangle$ by Lemma \ref{lem:densitysumdLa} $\delta({\rm L}_B)=\sum_{b\in B}{\rm d}({\rm L}_b)$ exists. Then for $z>1$ large enough we have $\delta({\rm L}_{B\cap(z,\infty)}) = \sum_{b\in B, b>z}{\rm d}({\rm L}_b)< \overline{\delta}(A)$. 
Now by assumption $\overline{\delta}(A\cap {\rm L}_b)=0$ for all $b\le z$, $b\in B$, and so 
\begin{align*}
\overline{\delta}(A) = \overline{\delta}(A\cap {\rm L}_{B\cap (z,\infty)}) \ \le \ \delta({\rm L}_{B\cap (z,\infty)})   < \overline{\delta}(A),
\end{align*}
a contradiction. Hence there exists $a\in A$ such that $A\cap {\rm L}_a$ has positive upper $\log$ density.
\end{proof}

\vspace{.5em}
\noindent
{\bf Theorem \ref{thm:ESSLchain}.} {\it If $\overline{\Delta}(A)>0$, then there is an infinite {\rm L}-divisibility chain $D\subset A$ of growth}
\begin{align*}
\limsup_{y\to\infty}\sum_{\substack{d\in D\\d\le y}}\frac{1}{\log\log y} \ \ge \ \frac{\overline{\Delta}(A)}{e^\gamma}.
\end{align*}
\begin{proof}
Take $\eps>0$. Without loss, we may suppose $A\subset [x_\eps,\infty)$ for $x_\eps$ sufficiently large, so that by Proposition \ref{prop:Lprimsup} $f(A')\le e^\gamma+\eps$ for all {\rm L}-primitive subsets $A'\subset A$. 

By definition of upper $\log\log$ density $\Delta:=\overline{\Delta}(A)>0$, there exists an unbounded sequence $(x_j)_{j=0}^\infty\subset\R$ such that for all $j\ge0$,
\begin{align}\label{eq:deltaAxj}
f(A\cap [1,x_j]) = \sum_{\substack{a\in A\\ a\le x_j}}\frac{1}{a\log a} > (\Delta-\eps)\log\log x_j.
\end{align}

Recall the {\rm L}-primitive generating set $\langle S\rangle=\{s\in S: s\notin {\rm L}_t \; \forall t<s,t\in S\}$ of a set $S\subset \N$ from Lemma \ref{lem:genset}. We partition $A = \bigcup_{i\ge0} A^i$ into a disjoint collection of {\rm L}-primitive subsets, where $A^0 = \langle A\rangle$ and inductively $A^l = \langle A\setminus \bigcup_{i<l} A^i\rangle$. By construction each $a=a_l\in A^l$ has a (finite) chain of {\rm L}-divisors $a_i\in A^i$ with ${\rm L}_{a_0}\supset \cdots \supset {\rm L}_{a_l}={\rm L}_a$. Also note $f(A^i)\le e^\gamma+\eps$ by assumption, so in particular $A^i$ has zero $\log\log$ density. Hence \eqref{eq:deltaAxj} implies each $A^i$ in $A=\bigcup_{i\ge0} A^i$ is non-empty. Next, define the subset $B = \bigcup_{j\ge0}B_j$ for
\begin{align*}
B_j \ : = \ A\cap [1,x_j]\setminus \bigcup_{1\le i<r_j} A^i, \qquad\text{where}\quad r_j := \frac{\Delta-2\eps}{e^\gamma+\eps}\log\log x_j.
\end{align*}
Note the sets $B_j$ are pairwise disjoint: Indeed, since $A = \bigcup_{i\ge0} A^i$, for each $j$ we have $A\cap [1,x_j]\subset \bigcup_{i<s_j} A^i$ for some finite $s_j$, as determined by $x_j$. Then since $(x_j)_j$ is unbounded, (passing to a subsequence) we have $r_{j+1} > s_j$ and so $A\cap [1,x_j]\subset \bigcup_{i<r_{j+1}}A^i$. Thus $B_j=A\cap [1,x_j]\setminus \bigcup_{i<r_j} A^i \; \subset\; \bigcup_{r_j\le i<r_{j+1}} A^i$ inherits disjointness from the $A^i$, as claimed.

Since $B = \bigcup_{j\ge0}B_j$ forms a disjoint union, for each $b\in B$ there is a unique index $J(b)$ such that $b\in B_{J(b)}$, that is,
\begin{align}
b \ \le \ x_{J(b)} \quad \text{and}\quad b\notin \bigcup_{i<r_{J(b)}} A^i.
\end{align}

In addition, $B$ has positive upper $\log\log$ density, since by definitions of $B$, $r_j$, and \eqref{eq:deltaAxj},
\begin{align*}
f(B\cap [1,x_j]) \ \ge \ f(B_j)
& \ \ge \ f(A\cap [1,x_j]) \ - \ \sum_{i<r_j}f(A^i)\\
& \ > \ (\Delta-\eps)\log\log x_j - r_j(e^\gamma+\eps) \ = \ \eps\log\log x_j.
\end{align*}
In particular $B$ has positive upper $\log$ density, so by Theorem \ref{thm:LDavenErdos} there exists an {\it infinite} {\rm L}--divisibility chain $D\subset B$. Since $D:=(d_k)_{k=0}^\infty$ is unbounded, (by passing to a subchain) we may assume $J(d_k) < J(d_{k+1})$ for all $k\ge0$. 
Recall each $a\in A^i$ is at the end of an ${\rm L}$-divisibility chain of length $i$. As $b\in B_{J(b)}$ and $B_j$ is contained in $\bigcup_{r_j\le i<r_{j+1}}A^i$, we infer each $d\in D\subset B$ is at the end of an ${\rm L}$-divisibility chain of length (at least) $r_{J(d)}$. Write it as $c_0^{(k)}\mid c_1^{(k)}\mid \cdots\mid c_{r_{J(d_k)}}^{(k)}=d_k$, with
\begin{align*}
{\rm L}_{c_0^{(k)}} \supset \cdots \supset {\rm L}_{d_k}.
\end{align*}
Now let $i_k$ be the least index such that $c_{i_k}^{(k)}>d_{k-1}$ and define
\begin{align*}
C := \{d_{k-1}< c_i^{(k)} \le d_k \ : \ k,i\ge0\} \; =\; \bigcup_{k\ge0}\{c_i^{(k)}  \ : \ i\in [i_k, r_{J(d_k)}]\} .
\end{align*}
We may assume $(r_j)_j$ grows fast enough, so that $\lfloor\eps\, r_{J(d_k)}\rfloor > d_{k-1}$.
Then the trivial bound $c_i^{(k)} > i$ implies $c_{\lfloor\eps \,r_{J(d_k)\rfloor}}^{(k)} > d_{k-1}$, and so $\lfloor\eps \,r_{J(d_k)}\rfloor \ge i_k$. Thus
\begin{align}\label{eq:supE}
\big|C\cap [1,x_{j(d_k)}]\big| 
\;\ge\; \big|C\cap [d_{k-1},d_k]\big|
\;\ge\; (1-\eps)r_{J(d_k)} = (1-\eps)\frac{\Delta-2\eps}{e^\gamma+\eps}\log\log x_{j(d_k)}.
\end{align}
Hence taking $\eps\to 0$ in \eqref{eq:supE} above gives $\limsup_{x\to\infty} \sum_{c\in C,c\le x}1/\log\log x \ge \Delta/e^\gamma$ as desired.

Finally note $C$ forms an infinite {\rm L}-divisibility chain: for each $k$ we have $c_j^{(k)}\in{\rm L}_{c_i^{(k)}}$ for all $i_k\le i<j$, in particular $d_k\in{\rm L}_{c_i^{(k)}}$.
Also $d_k\in{\rm L}_{d_{k-1}}$ since $D$ is an {\rm L}-divisibility chain, so there exist factorizations
\begin{align*}
d_k = gc_i^{(k)} = hd_{k-1},
\end{align*}
with $p(g)\ge P(c_i^{(k)})$ and $p(h)\ge P(d_{k-1})$. As $c_i^{(k)}> d_{k-1}$, we deduce $c_i^{(k)}\in{\rm L}_{d_{k-1}}$. Thus the $k$th and $(k-1)$th pieces of $C$ are linked together. Hence $C$ is indeed an {\rm L}-divisibility chain.
\end{proof}

\section{Closing remarks}

In this discussion, we attempt to sample just a few of the multitude of open questions that have quickly arisen in connection with the Erd\H{o}s primitive set conjecture. We have already described a few in the introduction, including Conjecture \ref{conj:ESS}, as well as whether $p=2$ is Erd\H{o}s strong. We also note recent work has studied variants of the problem in function fields $\mathbb{F}_q[x]$, see \cite{funcfield}, \cite{funcfield2}. In addition, it would be interesting to further extend the classical study of sets of (all) multiples and of primitive sets, e.g. see Hall \cite{Hsetmult} or Halberstam--Roth \cite[\S 5]{HalbRoth}, to sets of {\rm L}-multiples and {\rm L}-primitive sets.


We conclude with a related question of Banks and Martin, which offers a potential unified framework to view the results described in this article. For $k\ge1$, recall $\mathbb{N}_k=\{n : \Omega(n)=k\}$, in particular $\mathbb{N}_1=\mathcal{P}$. In 1993, Zhang \cite{zhang2} proved $f(\mathbb{N}_k)<f(\mathcal{P})$ for each $k>1$. Later Bayless, Kinlaw, and Klyve \cite{BKK} showed that $f(\N_2) > f(\N_3)$. Banks and Martin \cite{BM} predicted $f(\mathbb{N}_k)<f(\mathbb{N}_{k-1})$ for each $k>1$. 
In fact, they posed a vast generalization to Conjecture \ref{conj:EPS}.

\begin{conjecture}[odd Banks--Martin] \label{conj:BM}
Let $k\ge1$ and suppose $A$ is a primitive set with $\Omega(n)\ge k$ for all $n\in A$. Then for any set of odd primes $\mathcal Q$, we have
\begin{align}\label{eq:BMAQ}
f(A(\mathcal Q)) \ \le \ f\big(\mathbb{N}_k(\mathcal Q)\big).
\end{align}
Here $A(\mathcal Q)$ denotes the set of members of $A$ composed of primes in $\mathcal Q$.
\end{conjecture}

Banks and Martin managed to show \eqref{eq:BMAQ} in the special case when the set of primes $\mathcal Q$ is quite sparse, namely $\sum_{p\in \mathcal Q}1/p <1.74$ (even when $2\in \mathcal Q$). We note the original formulation of Conjecture \ref{conj:BM} included the cases $2\in \mathcal Q$, but this turns out to be false. Indeed, when $\mathcal Q=\mathcal P$ it was shown $f(\mathbb{N}_k) > f(\mathbb{N}_6)$ for each $k\neq6$ \cite{Lalmost}. In fact, numerical evidence suggests that in fact the reverse holds $f(\mathbb{N}_k)>f(\mathbb{N}_{k-1})$ for $k>6$. Nevertheless for $\mathcal Q=\mathcal P\setminus \{2\}$, the desired inequality $f(\mathbb{N}_k(\mathcal Q))<f(\mathbb{N}_{k-1}(\mathcal Q))$ holds up to at least $k=20$.

Observe that Theorem \ref{thm:Estrongodd} implies Conjecture \ref{conj:BM} in the special case $k=1$. Indeed, if $p\notin \mathcal Q$ then $A(\mathcal Q)_p=\emptyset$, so we deduce $f(A(\mathcal Q)) = \sum_{p\in \mathcal Q}f(A(\mathcal Q)_p) \le \sum_{p\in \mathcal Q}f(p) = f(\mathcal Q)$.
Moreover, if true, Conjecture \ref{conj:BM} implies Conjecture \ref{conj:ESS} of Erd\H{o}s--S\'ark\"ozy--Szemer\'edi. This follows by an argument similar to Theorem \ref{thm:ESS}, and using $f(\N_k(\mathcal Q))\to 1/2$ as $k\to\infty$ when $\mathcal Q=\mathcal P\setminus \{2\}$, see \cite[Corollary 4.2]{Lalmost}. We leave this to the interested reader.

\section*{Acknowledgments}
The author expresses deep gratitude to Carl Pomerance for many discussions, and to Paul Kinlaw and James Maynard for careful readings and feedback. The author more broadly thanks Tsz Ho Chan and Scott Neville for engaging conversations over the years. The author also thanks Andr\'as S\'ark\"ozy for bringing \cite{MathErdosII} and references therein to his attention, as well as Michel Balazard, Fran\c{c}ois Morain for the reference \cite{ErdosLimoges} on the Erd\H{o}s primitive set conjecture. 
The author was supported by the Clarendon Scholarship and Balliol College at the University of Oxford, as well as the European Research Council (ERC) under the European Union’s Horizon 2020 research and innovation programme (grant agreement No. 851318).

\bibliographystyle{amsplain}

\end{document}